\documentclass[11pt,a4]{amsart}

\usepackage{amsmath}
\usepackage{amssymb}
\usepackage{cancel}
\usepackage{graphicx}
\usepackage{url,hyperref}
\usepackage[style=american]{csquotes}
\usepackage{enumerate}
\usepackage[width=\textwidth]{caption}
\usepackage{subcaption}
\DeclareCaptionSubType*[arabic]{figure}
\usepackage[normalem]{ulem}

\usepackage{own-macros}

\begin{document}



\title[The asymmetry of complete and constant width bodies and the Jung constant]{The asymmetry of complete and constant width bodies in general normed spaces and the Jung constant}

\author{Ren\'e Brandenberg and Bernardo Gonz\'alez Merino}
\address{Zentrum Mathematik, Technische Universit\"at M\"unchen, Boltzmannstr. 3, 85747 Garching bei M\"unchen, Germany} \email{brandenb@ma.tum.de} \email{bg.merino@tum.de}


\thanks{The second author was partially supported by MINECO/FEDER project reference MTM2012-34037, Spain.}



\begin{abstract}
  In this paper we state a one-to-one connection between the maximal ratio of the circumradius and the diameter of a body (the Jung constant)
  in an arbitrary Minkowski space and the maximal Minkowski asymmetry of the complete bodies within that space.
  This allows to generalize and unify recent results on complete bodies and to derive a necessary condition on the unit ball of the space,
  assuming a given body to be complete. Finally, we state several corollaries, \eg concerning the Helly dimension or the Banach-Mazur distance.
\end{abstract}

\maketitle

\section{Introduction} \label{s:intro}

A bounded set in a Minkowski space (a finite dimensional
real normed space) is called (diametrically) complete
if it cannot be enlarged without increasing its diameter.
In Euclidean spaces and in planar Minkowski spaces the complete sets
are precisely the sets of constant width, while in arbitrary Minkowski spaces
every constant width set is complete, but the converse is not true in general
(\cf~\cite{Eg3}). Actually, the norms of Minkowski spaces in which 
all complete sets are of constant width are called perfect and
characterizing such spaces and norms is still a major open task in
convex geometry (\cf~\cite{Eg3,MSch}) and in functional analysis (%
\cite{CP} shows that norms of general Hilbert spaces are perfect, and \cite{MPP}
studies properties of perfect norms in infinite dimensional Banach
spaces).

There exists a rich variety of asymmetry measures for convex sets
(see \cite[Sect. 6]{Gr} for the possibly most comprehensive
overview) but amongst all, the one receiving most attention is the
\cemph{dred}{Minkowski asymmetry} $s(K)$, which may be expressed as the smallest
dilatation factor needed to cover the origin mirrored set $-K =
\{-x: x \in K\}$ of a set $K$ by a homothetic of $K$ itself. In
mathematical terms:

\begin{equation}
s(K) := \inf\{\rho > 0 : \exists \ c \in \R^n \text{ \st } -(K-c) \subset \rho (K-c)\},
\end{equation}
and if $-(K-c) \subset s(K) (K-c)$ then $c$ is called  the \cemph{dred}{Minkowski center} of $K$.

It is an easy consequence of the Blaschke Selection Theorem (see \eg~\cite{Sch})
that this infimum is attained and the same holds true for all other infima or suprema in the following.

It is well known \cite{Gr} and easy to show that $1 \le s(K) \le n$ with equality in the left inequality iff $K$
is symmetric and equality in the right inequality iff $K$ is a fulldimensional simplex.

For polytopes (presented by their vertices or their facets) the Minkowski asymmetry can be computed via
Linear Programming (see \cite{BeF} or \cite[Lemma 3.5]{BrK2}). Hence, involving this asymmetry
has not only theoretical interest but is useful in computations, too (\cf \cite{BrK}).

The question of finding the most asymmetric sets of constant width
is a well known topic of study
(see, \eg~\cite{GJ} for the Minkowski asymmetry or \cite[Theorem 56]{Egg} for the asymmetry of Besicovitch).
For the Minkowski asymmetry of constant width sets $K$ in Euclidean spaces recently the following
inequality has been derived in \cite{GJ}:

\begin{equation} \label{eq:GJ} 
  s(K) \le \frac{n+\sqrt{2n(n+1)}}{n+2},
\end{equation}

leaving the characterization of the equality case open, which then is solved separately in \cite{GJ3}:
Equality in \eqref{eq:GJ} holds, iff $K$ is a completion of an $n$-dimensional regular simplex.

This result was the starting motivation for the present paper.
In \cite{Sch2} Schneider says \enquote{as a rule, the first step $[\dots]$
consists in optimizing a proof of the inequality to make the identification of the equality
cases as easy as possible}. We believe our work leads into this direction.

Geometric inequalities relating different radii of convex bodies like the circumradius $R(K)$ and the diameter $D(K)$
form a central area of research in convex geometry. The upper bounds for the ratio of the circumradius and the diameter by
Jung \cite{Ju} (for Euclidean space)

\begin{equation} \label{eq:Jung}
  R(K)/D(K) \le \sqrt{\frac{n}{2(n+1)}}
\end{equation}

and Bohnenblust \cite{Bo} (for arbitrary Minkowski spaces)

\begin{equation} \label{eq:Bohnenblust}
  R(K)/D(K) \le \frac{n}{n+1}
\end{equation}


are famous, widely studied and have been object of many improvements and extensions (\eg in \cite{BM,Bo,BrK,BrK2,H,Ju}).

Explicitely mentioning \cite{BF}, in many classical works in
convexity, significant parts are devoted to geometric inequalities among the basic radii (see \cite{Bo,Le,Sa,St21}),
generalizations (\cite{BeHe2,Br2,Gon,H,Per}), and between radii and other functionals (\cite{BeHe1,BHCS,HeHC}).

In recent years both, Jung's inequality as well as the Minkowski asymmetry have drawn renewed attention.
Jung's inequality has been generalized and improved in several ways (see, \eg, \cite{BM,BrK,H}).
In \cite{Sch2} the Minkowski asymmetry has been used as a parameter for improving geometric
inequalities and in \cite{BeF,GuK1} especially as a connection between a geometric inequality and
its restricted version to symmetric sets.
Fundamental results used in the present paper stem from \cite{BrK2}. There, besides others, sharpened versions
of Jung's and Bohnenblust's inequalities involving the Minkowski asymmetry have been derived
(see Proposition \ref{BrK_Bohnenbl} below), which are of major interest here, too.

The main result we obtain is a one-to-one relation between the
maximal Minkowski asymmetry of complete bodies within a given
Minkowski space $\M^n = (\R^n,\norm)$ and the \cemph{dred}{Jung constant} $j(\M^n)$
of the same space, which measures the maximal ratio between the
circumradius and the diameter of arbitrary bodies in that space.

\begin{thm}\label{th:s_link}
For any Minkowski space $\M^n$ and $K$ complete within $\M^n$ it holds
\[s(K) \le \frac{j(\M^n)}{1-j(\M^n)},\]
and for any $K$ complete within $\M^n$ it holds $s(K)=\frac{j(\M^n)}{1-j(\M^n)}$,
iff $K$ is the completion of an $n$-simplex $S$ with circumradius-diameter ratio $j(\M^n)$.
\end{thm}

Defining the \cemph{dred}{asymmetry constant} $s(\M^n)$ of a Minkowski space $\M^n$
by $s(\M^n) := \max \{ s(K) \ :  \ K \text{ complete}\}$, Theorem \ref{th:s_link}
simply says that $s(\M^n) = \frac{j(\M^n)}{1-j(\M^n)}$ or $j(\M^n) = \frac{s(\M^n)}{s(\M^n)+1}$.


This does not only reprove inequality \eqref{eq:GJ} and its equality
case (as to be seen below in a simpler, less technical way), it
also generalizes it to arbitrary Minkowski spaces and therefore
allows to unify several results seemingly not related and thus
generalizing and strengthening them. The base of this unification
lies in the understanding of the direct interplay of two geometric
inequalities belonging to the Jung-constant and completeness, which
are studied 
in \cite{BrK} and \cite{BrK2}.



\bigskip

Before going into details, some necessary notation has to be stated:

For any $A \subset\R^n$ we write $\lin(A)$, $\aff(A)$, and $\conv(A)$ for the \cemph{dred}{linear}, \cemph{dred}{affine},
and \cemph{dred}{convex hull} of $A$, respectively,
and abbreviate by $[x,y] := \conv(\{x,y\})$ the line segment with endpoints $x,y\in\R^n$.

For any $\rho>0$ and $A, B\subset\R^n$ let $A+B :=\{a+b:a\in A,\,b\in B\}$ denote the
\cemph{dred}{Minkowski sum} of $A$ and $B$ and $\rho A := \{\rho a : a\in A\}$ the \cemph{dred}{$\rho$-dilatation} of $A$,
abbreviating $-A:=(-1)A$.

We use $\B$ to denote the unit ball of a Minkowski space $\M^n$ and in case the 
the Minkowski space is a Euclidean space, we write $\E^n=(\R^n,\norm_2)$ and $\B_2$ for the Euclidean unit ball.

A \cemph{dred}{(convex) body} is a set $K \subset \R^n$ that is convex and compact.
Let $\CK^n$ be the family of bodies $K \subset \R^n$, and $\CK^n_0$ its subset formed by
centrally symmetric sets, \ie, when $K=-K$.
All along the paper we refer by $S$ to an $n$-dimensional simplex and by $T$ if the simplex is regular (in the Euclidean sense).

For any two sets, we write $A \subset_t B$ (resp.~$A =_t B$) to denote that there exists a translation vector $c$
such that $A \subset c+B$ (resp. $A=c+B$) and abbreviate by $A \subset^{opt} B$ that $A \subset B$, but
$A \not \subset_t \rho B$ for any $\rho<1$.

Now, the \cemph{dred}{outer radius} or \cemph{dred}{circumradius} $R(K)$ of  $K \in \CK^n$ is
\[ R(K):=\inf\{\rho > 0 : K \subset_t \rho \B\} = \inf_{c \in \R^n} \sup_{x \in K} \norm[x-c]\]
(see \eg~\cite{GK}) and any $c\in \R^n$, \st~$K \subset c + R(K)\B$ is called a \cemph{dred}{circumcenter} of $K$.

Analogously, the \cemph{dred}{inradius} $r(K)$ is defined as $r(K) := \sup\{\rho > 0 : \rho \B \subset_t K\}$
and any $c\in \R^n$, \st~$c + r(K)\B \subset K$ is called an \cemph{dred}{incenter} of $K$.
The \cemph{dred}{diameter} $D(K)$ of $K$ is defined as $D(K) := 2 \sup_{x,y \in K} R([x,y]) = \sup_{x,y \in K} \norm[x-y]$ and
the \cemph{dred}{Jung ratio} as $j(K) := R(K)/D(K)$.
Using the \cemph{dred}{support function} of a convex body $K\in\CK^n$,
$h(K,\,\cdot\,):\R^n \rightarrow \R$, $h(K,u)=\sup_{x \in K} u^Tx$,
the width $w(K)$ of $K$ can be defined as $w(K) = \inf_{u \in \R^n\backslash\{0\}} \frac{h(K-K,u)}{h(\B,u)}
= \inf_{u \in \S^*} h(K-K,u)$, where $\S^*$ denotes the dual unit sphere.

A set $K \in \CK^n$ is said to be of \cemph{dred}{constant width}, if $w(K) = \frac{h(K-K,u)}{h(\B,u)}$,
independently of the choice of $u$. It is well known that $K$ is of constant width, iff $w(K)=D(K)$, or iff $K-K =D(K)\B$
\cite[(A)]{Eg3}.
A set $K$ is called \cemph{dred}{complete} if $D(K \cup \{x\}) > D(K)$ for all $x \notin K$ and a set $K^* \supset K$
is a \cemph{dred}{completion} of $K$ if $K^*$ is complete and $D(K^*) = D(K)$. If it also holds that $R(K^*)=R(K)$,
\ie~$K$ and $K^*$ possess a common circumball, 
then we say $K^*$ is a \cemph{dred}{Scott completion} of $K$.
Indeed, in \cite{Sc81} the existence of Scott completions is proved for the Euclidean case and
with almost the same proof \cite{Vr} shows the same for general Minkowski spaces.
We say a set $K$ is \cemph{dred}{pseudo-complete} if there exists $c\in\R^n$ \st
$c+(D(K)-R(K))\B\subset K\subset c+R(K)\B$ and for any circumcenter $c$ of $K$, the set
$K^+:=\conv\{K\cup(c+(D(K)-R(K))\B)\}$ is called a \cemph{dred}{pseudo-completion} of $K$.

The subsets of $\CK^n$ of all complete bodies and all bodies of constant width are abbreviated by $\Complete$ and $\ConstW$, respectively.

As mentioned in the introduction, it is well-known (see e.g.~\cite{Eg3}) that $\ConstW \subset \Complete$ and
$\ConstW = \Complete$ in the planar case or for $\E^n$.
Unfortunately, this is not the case in general Minkowski spaces when
$n\ge 3$ (as pointed out in \cite{Eg3} and \cite{MaWu}).
For instance, let $\M^n$ be a Minkowski space with indecomposable unit ball $\B$,
\ie~for all $K,L\in\CK^n$ with $\B=K+L$, there exist $\lambda,\mu\ge  0$, \st
$\lambda K =_t \mu L =_t \B$ (\cf~\cite{Sch}). An example for the latter is if $\B$ is the unit crosspolytope.
From the indecomposabilty it follows on the one side that $\ConstW=\{\rho \B + t : \rho \ge 0, t \in \R^n\}$
(as $K\in\ConstW$ iff $K-K=D(K)\B$) and on the other side
it is well known \cite{Eg2,So} and can easily be followed from Theorem
\ref{th:s_link} that $\Complete=\{\rho \B + t : \rho \ge 0, t \in \R^n\}$ holds, iff $\B$ is a
parallelotope (a highly decomposable set).

\bigskip

In \cite[p.~125]{Eg3} (for the Euclidean case) and \cite[Theorem 2]{MSch} (for general Minkowski spaces)
inequalities lower bounding the ratio between the inradius and the diameter of complete sets were derived.
Let $\M^n$ be a Minkowski space and $K\in\Complete$. Then
\begin{equation} \label{eq:MSchEgg}
  \frac{r(K)}{D(K)} \ge \frac{1}{n+1} \qquad \text{and} \qquad
  \frac{r(K)}{D(K)} \ge 1-\sqrt{\frac{n}{2(n+1)}}, \text{ if } \M^n=\E^n.
\end{equation}
Equality holds in the first, if
$K$ is an $n$-simplex and $\B=K-K$,
and in the second inequality, iff $K$ is a completion of a regular $n$-simplex.
In the proof of the general inequality above, Moreno and Schneider used that for any complete set $K$ it holds $r(K)+R(K)=D(K)$.

Involving the Minkowski asymmetry we will not only reprove the inequalities and the equation $r(K)+R(K)=D(K)$ above,
but reveal a deeper relation for these radii for complete sets which also allows to generalize the bound to arbitrary \emph{fixed}
Minkowski spaces:

\begin{thm} \label{thm:rD}
Let $\M^n$ be a Minkowski space and $K\in\Complete$. Then it holds 
\[(s(K)+1)r(K)=r(K)+R(K)=\frac{s(K)+1}{s(K)}R(K)=D(K),\]
and therefore
\[\frac{r(K)}{D(K)} = 1 - j(K) \ge 1 - j(\M^n) \]
with equality, iff $j(K) = j(\M^n)$.
\end{thm}

Moreover, the same relations also allow to formulate conditions on the space $\M^n$
if given a \emph{fixed} $K \in \CK^n$ chosen to be complete:

\begin{thm} \label{th:cp_nec_ball}
  Let $K\in\CK^n$ and 0 a Minkowski center of $K$. Then $K \in \Complete$ implies the following condition on the unit ball $\B$ of $\M^n$:
  \[K-K \subset D(K) \B \subset (s(K)+1)(K\cap(-K)).\]
  If in addition $K$ is a fulldimensional simplex, then either $\ConstW = \{\rho\B + t: \rho \ge 0, t \in \R^n\}$ or
    $K-K=D(K)\B$ (and therefore $K \in \ConstW$).
\end{thm}

\section{Pseudo-completeness and completeness}\label{Pseudo}

The following proposition is taken from \cite[Corollary 6.3]{BrK2}:

\begin{prop} \label{prop:chain}
For any Minkowski space and any $K\in\CK^n$ it holds
\begin{equation*}
\begin{split}
w(K) \leq (s(K)+1)r(K)\leq r(K)+R(K) 
\leq\frac{s(K)+1}{s(K)}R(K)\leq D(K) .
\end{split}
\end{equation*}
\end{prop}

Proposition \ref{prop:chain} also allows an immeadiate corollary,
summarizing two inequalities we need later.

\begin{cor}\label{BrK_Bohnenbl}
For any Minkowski space and any $K\in\CK^n$ it holds
\begin{equation} \label{eq:JungBrK}
j(K) \le \frac{s(K)}{s(K)+1},
\end{equation}
and
\begin{equation}\label{eq:Dr}
\frac{r(K)}{D(K)}\leq\frac{1}{s(K)+1},
\end{equation}
\end{cor}
In particular
\eqref{eq:JungBrK} sharpens \eqref{eq:Bohnenblust},
(but has been restricted to Minkowski space
from \cite[Theorem 4.1]{BrK2}), while \eqref{eq:Dr} generalizes and
sharpens an inequality for Euclidean space of Alexander (see
\cite{Alex} and \cf \cite[Corollary 6.4]{BrK2}).

Particularizing in $\E^n$, \eqref{eq:JungBrK}
also leads to a sharpening of
Jung's inequality \eqref{eq:Jung} (\cf \cite[Corollary 5.1]{BrK2}):

\begin{prop} \label{BrK_Jung}
 In Euclidean spaces $\E^n$ it holds
\begin{equation*} 
  j(K) \le \min\left\{\sqrt{\frac{n}{2(n+1)}},\frac{s(K)}{s(K)+1}\right\}
\end{equation*}
for any $K \in \CK^n$.
\end{prop}
One should remark that the inequality above is fulfilled with equality for both values inside the minimum
\emph{at the same time} iff $\conv\{T\cup \frac{1}{s(T^*)}(-T)\}\subset K\subset T^*$.

Directly contained within the inequality chain in Proposition
\ref{prop:chain} is the inequality
 \begin{equation}\label{eq:Rr<D} R(K)+r(K) \le D(K),\end{equation}
 which we want to investigate closer in the remainder of this section.
It has been proven in \cite{Sa} for the Euclidean case and $n=2$ and
in \cite{Br} for arbitrary $n$. Eventhough it follows from Proposition \ref{prop:chain}
for general Minkowski spaces we want to present an easy direct proof
from which we extract  the fact that any set attaining equality in \eqref{eq:Rr<D} must have
some common in- and circumcenters.

\begin{lem}\label{lem:DRrs}
 For any Minkowski space $\M^n$ and any $K \in \CK^n$ it holds
 $R(K)+r(K) \le D(K)$
 and if $R(K)+r(K)=D(K)$ any incenter of $K$ is also a circumcenter.
\end{lem}

\begin{proof}
  For showing the inequality, assume \Wlog that 0 is an incenter of $K$.
  Since there exists a point in $K$ which is at least at distance $R(K)$
  from the incenter, there must be some $p \in \bd(\B)$ and $\rho \ge R(K)$,
  \st $\rho p \in K$. But, since
  $r(K)\B \subset K$, we also have $r(K)(-p) \in K$ and therefore
  $D(K) \ge \norm[\rho p-(-r(K)p)] \ge R(K)+r(K)$.

  The concentricity statement then follows directly from the necessity of
  $\norm[\rho p-(-r(K)p)] = R(K)+r(K)$, which means there is no point at a distance bigger than
  $R(K)$ from any incenter.
\end{proof}

It was shown in \cite{MSch} (by combining results from \cite{Eg3}
and \cite{Sall}) that \eqref{eq:Rr<D} holds with equality for any $K
\in \Complete$. The converse is in general not true, but the equality case of
\eqref{eq:Rr<D} plays a major role in understanding completions and
inequalities concerning them.
The following lemma now completely characterizes this equality case 
(partially it is a direct consequence of Proposition \ref{prop:chain}).

\begin{lem}\label{lem:pseudo}
Let $\M^n$ be a Minkowski space and $K\in\CK^n$.
Then the following are equivalent:
\begin{enumerate}[(i)]
\item\label{ps1} $D(K)=r(K)+R(K)$,
\item\label{ps3} $K$ is pseudo-complete with respect to $\B$,
\item\label{ps4} $D(K)=(s(K)+1)r(K)$ (equality case in Theorem \ref{thm:rD}), and
\item\label{ps2} for every incenter $c$ of $K$ it holds $K-K\subset D(K)\B\subset(s(K)+1)((K-c)\cap(-(K-c)))$.
\end{enumerate}
\end{lem}
\begin{proof}[Proof of Lemma \ref{lem:pseudo}]
In view of Theorem \ref{thm:rD} either \eqref{ps1} or \eqref{ps4} implies
$R(K)=s(K)r(K)$ and therefore the other.

Since $D(K)=R(K-K)$ we have that \eqref{ps4} (or \eqref{ps1}) implies:
\[K-K\subset R(K-K)\B=D(K)\B=(s(K)+1)r(K)\B\subseteq (s(K)+1)(K-c).\]
However, since $K-K$ is symmetric, we then also have $K-K\subset(s(K)+1)(-(K-c))$ and therefore \eqref{ps2}.

Now, observe that \eqref{ps2} implies $D(K)\B\subset (s(K)+1)(K-c)$ and therefore
$r(K)\geq D(K)/(s(K)+1)$. Together with Proposition \ref{prop:chain} we obtain \eqref{ps4}.

Next we show \eqref{ps1} implies \eqref{ps3}:
From Lemma \ref{lem:DRrs} we know that \eqref{ps1} means that $K$ has a common in- and circumcenter $c$,
and 
therefore $c+(D(K)-R(K))\B\subset K \subset c+ R(K)\B$,
which means $K$ is pseudo-complete.

Finally, \eqref{ps3} together with \eqref{eq:Rr<D} directly implies \eqref{ps1}.


\end{proof}

\begin{rem} \label{rem:pc-c}
  \begin{enumerate}[a)]
  \item \label{rem:pc-c-a} Lemma \ref{lem:pseudo} implies that $K$ is pseudo-complete iff
    $K$ attains equality in \eqref{eq:Rr<D}. Also taking Proposition \ref{prop:chain} into account we see that
    pseudo-completeness also implies
    equality in \eqref{eq:JungBrK}, but the converse is in general not true (as \eg
    $\conv\{T\cup \frac{1}{s(T^*)}(-T)\}$ is not pseudo-complete in Euclidean space, but
    fulfills equality in Proposition \ref{BrK_Jung}).
  \item \label{rem:pc-c-b} As said above, if $K$ is complete it fulfills $D(K)=r(K)+R(K)$, and thus every complete $K$ is pseudo-complete.
  \end{enumerate}
\end{rem}

\medskip

The following Lemma collects some facts for pseudo-completions and completions.

\begin{lem} \label{lem:pseudo2}
  \begin{enumerate}[a)]
  \item If $K \in \CK^n$ and $K^+$ a pseudo-completion of $K$, then $R(K^+) = R(K)$ and $D(K^+)=D(K)$.
  \item If $K$ is pseudo-complete and $K^*$ a completion of $K$, then $R(K^*)=R(K)$, $r(K^*) = r(K)$ and $s(K^*)=s(K)$.
  \end{enumerate}
\end{lem}

\begin{proof}
  \begin{enumerate}[a)]
  \item   Without loss of generality we may assume that 0 is a circumcenter of $K$ and  $K^+=\conv\{(D(K)-R(K))\B\cup K\}$.
    Obviously $K^+ \subset R(K)\B$, which means by the monotonicity of $R$ that $R(K^+)=R(K)$. Moreover, since $K^+$ is pseudo-complete,
    we know from Lemma \ref{lem:pseudo} that $D(K^+) = r(K^+)+R(K^+)  = D(K) - R(K) + R(K^+) = D(K)$.
  \item Since both sets are pseudo-complete, we have $r(K)+R(K)=D(K) =D(K^*)=r(K^*)+R(K^*)$, which means because of set monotonicity that
    $R(K^*)=R(K)$ and $r(K^*) = r(K)$. Finally, $s(K^*)=s(K)$ follows from Lemma \ref{lem:pseudo} \eqref{ps4}.
  \end{enumerate}
\end{proof}

As an easy consequence we obtain a reproof of the existence of Scott completions for general Minkowski spaces only relying on the
existence of arbitrary completions:

\begin{cor} \label{cor:SVcompletion}
Let $\M^n$ be a Minkowski space with unit ball $\B$ and $K\in\CK^n$. Then there exists a
Scott completion $K^*$ of $K$.
\end{cor}

\begin{proof}
Let $K^+$ be any pseudo-completion of $K$ and $(K^+)^*$ be
any completion of $K^+$, which by \cite{Gro} always exists and obviously is a completion of $K$ too.
Now, from taking both parts of Lemma \ref{lem:pseudo2} together, we obtain $R(K)=R(K^+)=R((K^+)^*)$.
\end{proof}

Even though every complete body $K$ is pseudo-complete, the converse is in general
not true. Because of this the case in which $K$ is a pseudo-complete simplex deserves special attention.

\begin{lem}\label{lem:-SinS}
Any simplex $S$ has a unique Minkowski center $c$ and  $S-S$ touches  all
facets of $(n+1)((S-c)\cap(-(S-c)))$ with a full facet of its own.
\end{lem}

\begin{proof}
The uniqueness of the  Minkowski center $c$ of any simplex follows directly from the optimality conditions for containment
under homothetics \cite[Theorem 2.3]{BrK}.
Now \Wlog, let $S$ be regular and $c=0$. Moreover, suppose that $S=\conv\{p^1,\dots,p^{n+1}\}$ and let $F_i=\conv\{p^k:k\neq i\}$
denote the opposing facet to $p^i$ in $S$.

Then for any fixed $j \in [n+1]$ it holds $F_j-p^j$ is a facet of $S-S$ with vertices $p^i-p^j$, $i \in [n+1]$, $i \neq j$. Now,
$-p^j = \sum_{k \ne j} p^k \in nF_j$ and thus $p^i-p^j \in (n+1)F_j$, $i \in [n+1], i \neq j$,
and since $p^i-p^j \in (n+1)(-F_i) \subset (n+1)(-S)$, we obtain $F_j-p^j \subset (n+1)(F_j \cap(-S))$
\end{proof}

\begin{cor}\label{c:simplexCom}
Let $\M^n$ be a Minkowski space and $S$ an $n$-simplex. Then the following are equivalent:
\begin{enumerate}[(i)]
\item\label{com3} $S$ is complete with respect to $\B$.
\item\label{com1} $D(S)=r(S)+R(S)$.
\item\label{com2} $S-S\subset D(S)\B\subset(n+1)((S-c)\cap(-(S-c)))$, where $c$ is the (unique) incenter of $S$.
\item\label{com4} $j(S)=n/(n+1)$ (which means equality in \eqref{eq:Bohnenblust}).
\end{enumerate}
\end{cor}

\begin{proof}
It follows from Lemma \ref{lem:pseudo} that \eqref{com1} and \eqref{com2} are equivalent and that \eqref{com3} implies either of them.
That \eqref{com2} characterizes the equality case in \eqref{eq:Bohnenblust} (and thus is equivalent to \eqref{com4}) is shown in \cite{Le}.

In order to show \eqref{com2} implies \eqref{com3}, we may assume \Wlog that $c=0$. Now, let $x \not\in S$
and denote the facet of $S$ separating $x$ from $S$ by $F_i$ and the vertex of $S$ not in $F_i$ by $p^i$. Then
 $p^i \in -nF_i \subset \bd(-nS)$ and $x \not\in S$. Thus $p^i -x \not\in -(n+1)S \supset D(S)\B$. Hence
 $D(S \cup \{x\}) > D(S)$, proving that $S$ must be complete.
\end{proof}

Putting all together we are ready to prove Theorem \ref{th:cp_nec_ball}.

\begin{proof}[Proof of Theorem \ref{th:cp_nec_ball}]
  The statement about the unit balls belonging to general complete $K$ follows directly
  from Lemma \ref{lem:pseudo} in conjunction with Remark \ref{rem:pc-c} \eqref{rem:pc-c-b}.

  Now, consider the part about the case when $K$ is an $n$-simplex. 
  Since $s(K)=n$ we have $K-K \subset D(K)\B \subset (n+1) (K \cap(-K))$. Assuming $K' \in \ConstW$ with $D(K')=D(K)$ it
  follows $K'-K'=D(K)\B \supset K-K$ and if
  $K'-K' \neq K-K$ there exist $x \in K' \setminus K$. Using the same argument as in the step from \eqref{com2} to\eqref{com3}
  in Corollary \ref{c:simplexCom}, we see that
  $p^i-x \in K'-K' \setminus (n+1)(K \cap (-K))$, contradicting $K'-K' = D(K)\B \subset (n+1)(K \cap (-K))$.
  Thus $K-K = D(K)\B$ or $\ConstW = \{\rho \B +t: \rho \ge 0, t\in \R^n\}$.
\end{proof}

\begin{rem}
  \begin{enumerate}[a)]
  \item While $\ConstW = \{\rho \B +t : \rho \ge 0, t\in \R^n\}$ obviously implies $\max_{K \in \ConstW} s(K) = 1$,
    it follows from the last lemma in \cite{Eg2} that at least for $n=3$ in case of a simplex being of constant width,
    there exists another simplex $S'$, \st $S'$ is complete with respect to $\B$
    but not of constant width.
    Thus in both cases with a complete simplex, it holds $\Complete \neq \ConstW$ (at least) in 3-space,
    but while in the first $\max_{K \in \ConstW} s(K) = 1$ in the latter $\max_{K \in \ConstW} s(K) = s(\M^n)=n$.
  \item One should recognize that the existence of a simplex $S \subset \R^n$ such that Property \eqref{com2} of Corollary \ref{c:simplexCom}
    is fulfilled with $\B$ being a regular crosspolytope
    is by \cite{Dol} equivalent to the existence of an $(n+1)$-dimensional Hadamard matrix.
  \end{enumerate}
\end{rem}





\begin{lem}\label{l:JungComplete}
Let $\M^n$ be a Minkowski-space. Then
\begin{equation} \label{eq:CompleteJung}
j(\M^n)=\max_{K\in\Complete} j(K)
\end{equation}
and for any $K \in \CK^n$ with $j(K)=j(\M^n)$ and any completion $K^*$ of $K$
there exists an $n$-simplex $S \subset K$ 
\st $D(S)=D(K)=D(K^*)$ and $R(S)=R(K)=R(K^*)$.
\end{lem}

\begin{proof}
Let $K\in\CK^n$  with $j(K)=j(\M^n)$ and $K^*$ an arbitrary completion of $K$. 
We obtain $D(K^*)=D(K)$ and $R(K^*) \ge R(K)$ and thus $j(K^*) \ge  j(K) = j(\M^n)$, which means $j(K^*) = j(\M^n)$
and therefore $R(K^*) = R(K)$.

Now, due to Helly's theorem there always exists an $n$-simplex $S \subset K$,
\st $R(S)=R(K)$ (\cf~\ref{lem:core}) and surely it holds $D(S) \le D(K)$ and therefore $j(S) \ge  j(K) = j(\M^n)$.
Again it directly follows $j(S)=j(\M^n)$ and $D(S)=D(K)$.
\end{proof}

Observe that we implicitely showed in the proof of Lemma \ref{l:JungComplete} that every completion of a set $K$
attaining equality in \eqref{eq:CompleteJung} is a Scott completion.


Lemma \ref{l:JungComplete} implies that there always exists some simplex $S$ with
$j(S)=j(\M^n)$. On the other hand $j(K) \ge 1/2$ for all $K$, with equality if $K=-K$.
Hence $j$ may be understood a kind of asymmetry measure, too -- one that depends on
the Minkowski space it is measured in. In this it differs from the Minkowsky asymmetry, but
measuring asymmetry with respect to the way distances are measured within the space
makes a certain sense to us.

\bigskip

Taking together the above results we may now prove Theorem \ref{th:s_link} and Theorem \ref{thm:rD}.

\begin{proof}[Proof of Theorems \ref{th:s_link} and \ref{thm:rD}]

It follows from Remark \ref{rem:pc-c} \eqref{rem:pc-c-a} that $j(K)=s(K)/(s(K)+1)$ for all $K\in\Complete$.
Now, maximizing over $K\in\Complete$ on both sides, we obtain from \eqref{eq:CompleteJung} that $j(\M^n)=s(\M^n)/(s(\M^n)+1)$.

The characterization of the equality case then is a direct corollary of that in Lemma \ref{l:JungComplete}.

Every $K\in \Complete$ implies the conditions in Lemma
\ref{lem:pseudo} and thus equality in all but the first inequality
of Proposition \ref{prop:chain}. The second follows from
$D(K)=(s(K)+1)r(K)$ and $s(K)\leq s(\M^n)$, with equality iff
$s(K)=s(\M^n)$, which by $j(K)=s(K)/(s(K)+1)$ and
$j(\M^n)=s(\M^n)/(s(\M^n)+1)$ occurs iff $j(K)=j(\M^n)$.

\end{proof}

\begin{rem} \label{rem:equal}
Surely Theorem \ref{th:s_link} holds valid if we enlarge the class of complete
to pseudo-complete sets. Indeed the proof only relies in the equality condition
of \eqref{eq:Rr<D}.

Moreover, an immediate consequence of Theorem \ref{th:s_link} is that the Jung
ratio of a (pseudo-) complete set only depends on the Minkowski space in the
way that the set has to be  (pseudo-) complete, but not in its value. This
means, if a set $K$ is  (pseudo-) complete in two different spaces
the Jung ratio $j(K)$ stays constant. (Remember that we may interpret Condition \eqref{ps2} of
\ref{lem:pseudo} as a characterization of the unit ball of the spaces in which $K$ is pseudo-complete.)
\end{rem}


\bigskip

Applying Theorem \ref{th:s_link} we obtain some direct but important corollaries.

Corollary \ref{cor:GJ_ExtJ} does not only imply \eqref{eq:GJ}, but also answers the question in \cite{GJ}
(which was answered separately in \cite{GJ3}) about a characterization of the \enquote{only if case} (and so does Theorem
\ref{th:s_link} even for general Minkowski spaces). Recall that if $n=2$ then $T^*$
(a completion of the regular simplex) equals
the \emph{Reuleaux triangle} and in case that $n=3$ has as particular examples the
two \emph{Meissner bodies} (see \cite{BF}
for a detailed construction and basic properties).


\begin{cor}\label{cor:GJ_ExtJ}
In Euclidean space it holds
\[ s(\E^n)=\frac{j(\E^n)}{1-j(\E^n)}=\frac{n+\sqrt{2n(n+1)}}{n+2}, \]
and for $K \in \ConstW$ it holds $s(K)=s(\E^n)$ if and only if $K$ is
the completion $T^*$ of a regular simplex.
\end{cor}

\begin{proof}
The first statement follows directly from combining Jung's inequality
\eqref{eq:Jung} and Theorem \ref{th:s_link}, the second afterwards
from combining Proposition \ref{BrK_Jung}, Theorem \ref{th:s_link} and Lemma \ref{l:JungComplete}.
\end{proof}

\bigskip

Similar studies have been recently done  in \cite{Ji} for sets of revolution (around an axis of symmetry).
As corollaries of \eqref{eq:GJ} and Theorem \ref{th:s_link} we provide simplified proofs for \cite[Theorem 1 and Theorem 2]{Ji}.


\begin{cor}\label{p:Ji}
  Let $K \in \CK^n$ be a body of revolution in $\E^n$. Then
  \[
  s(K)\le 2\quad\text{and}\quad s(K) \le \frac{\sqrt{3}+1}{2},\text{ if } K\in\Complete,
  \]
  with equality in the first inequality iff $K$ is the body of revolution of an isosceles
  triangle and in the second iff $K$ is the body of revolution of a Reuleaux triangle.
\end{cor}

\begin{proof}
Let $e^i$ be the $i$-$th$ unit vector, $L=\lin(\{e^1\})$, $F=\lin(\{e^1,e^2\})$, and assume $K$ to be a body of revolution
around $L$. Furthermore by $K|F$ we denote the orthogonal projection of $K$ onto $F$.

First observe that in general for any body of revolution $K$ around $L$ it obviously holds
that $-K \subset_t \rho K$ iff  $-K|F \subset_t \rho K|F$ for all $\rho \ge 0$. Thus $s(K) = s(K|F)$.

Now $K|F$ is always a 2-dimensional set and thus
$s(K|F) \le 2$, with equality iff $K|F$ is a triangle (axially symmetric with respect to $L$ because of $K$ being a body of revolution).
Hence $s(K|F)=2$ holds iff $K|F$ is an isosceles triangle possessing one vertex in $L$ and the
other two symmetric with respect to $L$.

Finally, assuming that $K$ is of constant width also  $K|F$ is of constant width and
it follows from Corollary \ref{cor:GJ_ExtJ}
\[s(K|F) \le s(\E^2) = \frac{\sqrt{3}+1}{2}\]
and that equality holds iff $K|F$ is the completion of a regular simplex -- \ie it must be a Reuleaux triangle --
symmetric with respect to $L$.
\end{proof}

\begin{rem}
One may easily extend Corollary \ref{p:Ji} by assuming that $K$ is a
body of revolution around an $i$-dimensional subspace $L$. More formally let $L^\bot$ denote the orthogonal
subspace of $L$, then $K = \{x|L + \norm[x-x|L]_2 (\B_2 \cap L^\bot) : x \in K\}$.
In this case $s(K) \le i+1$, with equality for any set of revolution of an
$(i+1)$-dimensional simplex symmetric with respect to $L$.
Analogously, if $K\in\ConstW$, again the bound may be improved using Corollary \ref{cor:GJ_ExtJ}
to $s(K)\le s(\E^{i+1})=\frac{i+1+\sqrt{2(i+1)(i+2)}}{i+3}$, with equality for
the body of revolution around $L$ of any completion $K^*$ of an $(i+1)$-dimensional regular simplex
with $i$ vertices in $L$ and the other two symmetric with respect to $L$, \st $K^*$ keeps symmetric around $L$
(which can be ensured by taking $\frac12(K_1^*+K_2^*)$, if $K_1^*$ is an arbitrary completion of the regular simplex within
any $(i+1)$-space containing $L$ and $K_2^*$ is its mirrored with respect to $L$.
\end{rem}

\bigskip

In a Helly-type theorem there is usually a family of objects $\CF$, a property $\Pi$ and a number $h \in\N$, \st
if every subfamily $\CS$ of $\CF$ with $|\CS| = h$ possesses property $\Pi$, then $\CF$ has property $\Pi$.
The classical theorem of Helly deals with the case when $\CF$ is a finite
family of convex sets in $\R^n$, $h = n + 1$, and the property $\Pi$ is that the sets should have a non-empty
intersection. Helly-type theorems play a central role in convex geometry (see, \eg, \cite{BKP,BKP2,DGK} for
an overview on this topic). 

A remarkable special case is when $\CF$ only consists of translates of one fixed $K\in\CK^n$ \cite{SzN}.
The \cemph{dred}{Helly dimension} $\him(K)$ of a set $K\in\CK^n$ is defined as
the smallest positive integer number $k$, \st whenever we consider a family of indices $I\neq\emptyset$ with
$\bigcap_{i\in J}(x_i+C)\neq\emptyset$, for any $J\subset I$, $|J|\leq k+1$, $x_i\in\R^n$ and $i\in I$,
it already follows $\bigcap_{i\in I}(x_i+C)\neq\emptyset$. For a Minkowski space $\M^n$, we also use
$\him(\M^n):=\him(\B)$.

The following lemma generalizes \cite[Lemma 2.2]{BrK} (with almost the same proof).

\begin{lem} \label{lem:core}
  Let $\him(\M^n)=k \in [n]$. Then $R(K)=\max\{R(L): L \subset K, |L| \le k+1\}$ for all $K \in \CK^n$.
  Furthermore, if $\dim(K) \ge k$, then there always exists a $k$-simplex $S \subset K$, \st
  $R(S)=R(K)$.
\end{lem}

\begin{proof}
  Let $R_k(K) := \max\{R(L): L \subset K, |L| \le k+1\}$
  (this is called the \cemph{dred}{$k$-th core radius} of $K$ in \cite{BrK}).
  Surely, for any $L \subset K$ it follows $R(L) \subset R(K)$ and therefore $R_k(K) \le R(K)$.

  Now, by definition of $R_k(K)$ any $L \subset K$ with $|L| \le k+1$ can be covered by a copy of $R_k(K)\B$.
  Hence for all such $L$ it holds $\bigcap_{x\in L} (x - R_k(K)\B) \neq \emptyset$. However, by definition of
  the Helly dimension and the compactness of $K$ this means that
  $\bigcap_{x \in K} (x - R_k(K)\B) \neq \emptyset$ and therefore that $K$ can be covered by a copy of
  $R_k(K)\B$. This means $R_k(K) \ge R(K)$ and altogether it follows $R(K)=R_k(K)$.

  Applying Helly's theorem within $\aff(L)$ we may always assume that the set $L \subset K$ with
  $|L|=k+1$ and $R(L)=R(K)$ is affinely independent. Hence, if $|L| \le k \le \dim(K)$,
  we can complete $L$ to the vertex set of a $k$-simplex.
\end{proof}

In \cite[Corollary 1]{BM} Boltyanski and Martini improved Bohnenblust inequality by showing
the following:\footnote{Observe that the authors used the equivalent notion of minimal dependance.}

\begin{prop} \label{p:BM}
For any Minkowski space $\M^n$ it holds
\begin{equation*}
j(\M^n) \le \frac{\him(\M^n)}{\him(\M^n)+1}.
\end{equation*}
\end{prop}

Applying Proposition \ref{p:BM} one may also obtain that the asymmetry constant and
the Helly dimension of any Minkowski space directly bound each other. Moreover, from applying
Lemma \ref{lem:core} also the equality case can be characterized.

\begin{cor} \label{cor:him}
Let $\M^n$ be a Minkowski space. Then
\[\lceil s(\M^n) \rceil \le \him(\M^n)\]
and
$s(\M^n)=\him(\M^n)$ iff there exists a $\him(\M^n)$-dimensional simplex $S$, \st
$s(S)=s(S^+)=s(S^*)$ for all of its completions $S^*$. Moreover, in that case it holds
 $S=S^+\cap\aff(S)$.
\end{cor}

\begin{proof}
Since the function $f(x)=x/(x+1)$ is increasing whenever $x>0$, the claimed inequality
follows directly from combining Proposition \ref{p:BM} with Theorem \ref{th:s_link}.

Considering the equality case,
there exists some $K \in \Complete$ with $s(K)=s(\M^n)$ and it follows from Lemma \ref{lem:core}
that there exists a $\him(\M^n)$-dimensional simplex $S \subset K$, \st $R(S)=R(K)$. Now, $s(K)=s(\M^n)$
implies $j(K)=j(\M^n)$.
Similar to the proof of Lemma \ref{l:JungComplete}, it follows that $j(S)=j(K)$, $D(K)=D(S)$, and
therefore that $K$ is a completion of $S$.
Altogether, $s(S)=s(K)$ is equivalent to $s(\M^n)=\him(\M^n)$.
Finally, since surely $S \cup (D(S)-R(S))\B \subset -s(S)(S \cup(D(S)-R(S))\B)$ in general it holds
$s(S^+) \le s(S)$. From $s(S^+)=s(S)$ we then obtain that $s(S^+\cap\aff(S))=s(S) = \dim(S)$
and since $R(S)=R(S^+)$ that $S^+\cap\aff(S)=S$.
\end{proof}

\begin{example}
Considering \eg $\E^n$, the inequality in \ref{cor:him} can be strict.

On the other hand, for any given $k \in [n]$ the inequality is sharp with
$s(\M^n)=\him(\M^n)=k$, for instance, if we take $K = S^k \times [0,1]^{n-k}$, where $S^k$ denotes a
$k$-dimensional simplex, and $\B=K-K$.

Especially, it is known that $\him(\M^n)=1$ iff $\B$ is a parallelotope (see~\cite{SzN}) and this again is eqivalent to
$\Complete=\{\rho \B + t : \rho \ge 0, t \in \R^n\}$ (follows from \cite{Eg3} together with \cite{So}).
Corollary \ref{cor:him} now says that all this is again equivalent with $s(\M^n)=1$ or $j(\M^n)=1/2$.

In case of $s(\M^n)=n$, \ie~by Corollary \ref{c:simplexCom} that there exists a simplex $S$ with incenter $c$, \st
$S-S \subset D(S)\B \subset(n+1)((S-c)\cap(-(S-c)))$ (and $S=S^+=S^*$), Corollary \ref{cor:him} implies $\him(\M^n)=n$.

Finally, let $\M^3$ be the space whose unit ball is the hexagonal
prism $\B=(T^2-T^2) \times [-1,1]$, where $T^2$ is a 2-simplex.
Then $K=T^2 \times [-1/2,1/2] \in \ConstW$ and since $s(K)=2=\him(\M^3)$ it follows $s(\M^3)=2$.
Now, $S=T^2 \times \{0\} \subset K$, \st $K=S^*$ is the unique completion of $S$ and $s(S)=2$. However,
since $R(S)=\frac23$ and $D(S)=1$ it follows $S^+ = \conv(T^2 \times [-1/3,1/3]) \neq S^*$.
\end{example}

\bigskip

\section{Completeness, Minkowski asymmetry, and Banach-Mazur distance}

We recall that the Banach-Mazur distance between two full-dimensional sets $K,C\in\CK^n$ is defined as
\[
d_{BM}(K,C) := \inf\{\rho>0 : K \subset_t A(C) \subset_t \rho K, \text{ with } A \text{ a
 linear map} \}.
\]
Gr\"unbaum in \cite{Gr} mentioned (but not proved) that $s(K)=\min_{\B\in\CK^n_0}d_{BM}(K,\B)$.
For completeness reasons we present a proof of
this result and point out that this minimum is attained if $\B=K-K$.

\begin{prop}\label{p:BMdistance}
For any $K \in \CK^n$ it holds
\[s(K)=\min_{\B\in\CK^n_0}d_{BM}(K,\B)=d_{BM}(K,K-K).\]
\end{prop}

\begin{proof}
Let $\B\in\CK^n_0$. If $\B\subseteq K\subseteq t\B$ with $0<t<s(K)$, then
$K\subseteq t\B=-t\B\subseteq -tK$, a contradiction with the
definition of $s(K)$. Thus $s(K)\leq\min_{\B\in\CK^n_0}d_{BM}(K,\B)$.

Conversely, since $K\subset -s(K)K$ it follows
\[
\frac{1}{s(K)+1}(K-K) \subset K\subset \frac{s(K)}{s(K)+1}(K-K)
\]
and thus $d_{BM}(K,K-K)\leq\frac{s(K)}{s(K)+1}(s(K)+1)=s(K)$, which finishes the proof.
\end{proof}

The following result is taken from \cite[Theorem 2]{GJ}. Even so it is a direct corollary of
Proposition \ref{p:BMdistance}, it gives a more accurate
description of the Banach-Mazur distance
when restricted to completions in $\E^n$.

\begin{prop}\label{cor:GJdistance}
Let $K\in\ConstW$ in $\E^n$. Then $s(K)=d_{BM}(K,\B_2)$.
\end{prop}

Now we state a corollary which generalizes Proposition \ref{cor:GJdistance}
to arbitrary Minkowski space (and sharpens Proposition \ref{p:BMdistance}).

\begin{cor}\label{cor:BMcompleteDistance}
Let $\M^n$ be a Minkowski space and $K$ pseudo-complete. Then $s(K)=d_{BM}(K,\B)$.
\end{cor}

\begin{proof}
If $K$ is pseudo-complete, we know from Lemma \ref{lem:DRrs}
that it has concentric in- and circumball and from Lemma \ref{lem:pseudo2} that
$r(K)s(K)=R(K)$.
Now, using Proposition \ref{p:BMdistance} we obtain
$d_{BM}(K,\B) \le \frac{R(K)}{r(K)} = s(K) = \min_{\B\in\CK^n_0}d_{BM}(K,\B)$ and therefore
$s(K)=d_{BM}(K,\B)$, as required.
\end{proof}

Similarly as we have done in Remark \ref{rem:equal} for Theorem \ref{th:s_link},
we should observe that the above corollary shows that the distance of any (pseudo-)complete set and
the unit ball of the space does not depend on which unit ball is chosen, as long as the set stays (pseudo-)complete.

\bigskip

Gr\"unbaum in \cite{Gr} pointed out that the \enquote{supermaximality property} (we
changed the original \enquote{superminimality} because it matches better with our
definition of asymmetry) is a very natural property which should be
true for a nice asymmetry measure: (A) an asymmetry $\overline{s}$
satisfies the \cemph{dred}{supermaximality property} if
$\overline{s}(K+L)\leq\max\{\overline{s}(K),\overline{s}(L)\}$,
$K,L\in\CK^n$.
The asymmetry of Minkowski, \eg, satisfies the maximality property.
This property characterizes simplices to be the most asymmetric sets for any
asymmetry measure in the plane (\cf~\cite{Gr}), but if this is true in higher dimension is still unknown.

Moreover, since simplices may be (almost) subdimensional and may even converge towards a line segment,
it is not incontrovertible
if this is a \enquote{very natural property}. As already mentioned above one may
also interpret $j$ as an asymmetry measure and argue that the sets with  $j(K)=j(\M^n)$
are somehow most asymmetric and the same can be done with the width-inradius ratio
or similar coefficients.

As a second property possibly to be fulfilled by a reasonable
asymmetry measure Gr\"unbaum suggested the equality case of the
supermaximality property: (B) if
$\overline{s}(K+L)=\max\{\overline{s}(K),\overline{s}(L)\}$ then
$K=-K$, $L=-L$ or $K$ and $L$ are homothetics of each other.
However, he did not even clarify if (B) is true for the Minkowski
asymmetry, his favourite asymmetry measure. Maybe that was the
reason that he also considered the following condition: (B') if
$K\in\CK^n$ and $\overline{s}$ fulfills the supermaximality, then
$\overline{s}(K)=n$ iff $K$ is an $n$-simplex. Surely, this
condition is fulfilled by the Minkowsky asymmetry. Moreover, (B')
implies (B) when restricted to
$\overline{s}(K+L)=\max\{\overline{s}(K),\overline{s}(L)\}=n$. In
fact, assuming \Wlog~that $\overline{s}(K+L)=\overline{s}(K)=n$, it
follows from (B') that $K$ and $K+L$ are simplices and from the
indecomposability of simplices that $K$ and $L$ must be even
homothetic.

We observe that property (B) is false for the Minkowski asymmetry:

\begin{rem}
  For any $n \in \N$ consider
  the  $n$-dimensional regular simplex $T$, any Minkowski space $\M^n$ in which $T^+$
  is not complete (\eg $\E^n$), as well as a completion $T^*$ of $T$. Since
  $s(T^+)=s(T^*)$
  and \[R(T^++T^*)+r(T^++T^*)=R(T^+)+r(T^+)+R(T^*)+r(T^*)=D(T^+)+D(T^*)=D(T^++T^*),\]
  Theorem \ref{thm:rD}
  implies that $s(T^++T^*)=R(T^++T^*)/r(T^++T^*)=s(T^+)=s(T^*)$,
  but obviously neither are $T^+,T^*$ symmetric nor homothetics of each other, which contradicts
  property (B).

  With the same arguments one can also show that property (B) does not even hold for $n\ge 3$ when
  we restrict to $\ConstW$ choosing instead of $T^+$ and $T^*$ two different completions
  of $T$ (such as the two 3-dimensional Meissner bodies in $\E^3$ or their bodies of evolution in higher dimensions).
  In 2-space, the Reuleaux triangle is the unique completions of $T$, but one may easily do
  the counterproof as above with two less asymmetric bodies of constant width.
\end{rem}




\bibliographystyle{amsplain}

\end{document}